\documentclass{amsart}

\usepackage{amsfonts, amsbsy, amsmath, amsthm, amssymb, latexsym, enumerate,verbatim}

\numberwithin{equation}{section}

\newtheorem{propo}{Proposition}[section]

\newtheorem{lemma}[propo]{Lemma}

\newtheorem{theo}[propo]{Theorem}

\newcommand{\ld}{,\ldots ,}

\newcommand{\ra}{ \rightarrow }

\newcommand{\lan}{ \langle }
\newcommand{\ran}{ \rangle }

\newcommand{\Irr}{\mathop{\rm Irr}\nolimits}

\newcommand{\ZZ}{\mathbb{Z}}

\newcommand{\al}{\alpha}
\newcommand{\be}{\beta}

\newcommand{\ep}{\varepsilon}

\newcommand{\lam}{\lambda }
\newcommand{\om}{\omega}
\newcommand{\Om}{\Omega}

\newcommand{\si}{\sigma }

\def\d12{{_{12}}}

\def\acf{{algebraically closed field }}

\def\ccc{{constituent }}

\def\f{{following }}

\def\ii{{if and only if }}

\def\ir{{irreducible }}

\def\irt{{irreducible. }}
\def\irr{{irreducible representation }}

\def\itf{{It follows that }}

\def\mult{{multiplicity }}

\def\rep{{representation }}
\def\reps{{representations }}
\def\rept{{representation. }}

\def\2k{2^k}

\newcommand{\bp}{\begin{proof} }
\newcommand{\enp}{\end{proof}}
\newcommand{\bl}{\begin{lemma}\label}
\newcommand{\el}{\end{lemma}}

\newcommand{\med}{\medskip}

\renewcommand{\setminus}{\smallsetminus}

\def\hw{highest weight }

\begin{document}

\title[ Lie algebra representations  with maximum weight  multiplicity  2]
{Irreducible representations of simple  Lie algebras   with maximum weight  multiplicity  2  }
\author{ A.~E.~Zalesski }
\address{Departmento  de Mathem\'atica, Universidade de Bras\'ilia, Brasilia-DF, 70910-900, Brazil  }
\email{alexandre.zalesski@gmail.com}
 

\begin{abstract}
We determine   
the irreducible representations of simple Lie algebras with maximum weight  multiplicity  2. 
\end{abstract}


\date{}

\maketitle


\section{Introduction}

There is a significant number of research works devoted to the classification of \ir \reps of Lie algebras  whose all weights are of \mult 1, see 
\cite[\S 6]{Seitzclass}, \cite{BZ,Ki,Ki2}, \cite{Hw}.  
The result which is now well known,  has many applications. The natural question on determining the \ir \reps whose all weights are of \mult at most 2
has arisen in a recent book by N. Katz and Pham Tiep \cite[Section 3.2]{KT} for applications to the theory of monodromy groups of hypergeometric sheaves, a certain aspect of number theory. Author's results there are incomplete, they 
only considered the self-dual \reps of simple algebras of type $D_n$, $n$ odd, $E_6$ and $A_n$, $n>1$. 
 
In this paper we complete the classification of  \ir representations of simple Lie algebras  with maximal weight \mult equal to 2. 

\begin{theo}\label{th1} Let $L$ be a simple Lie algebra over the complex numbers and $\phi$ is an \irr of $L$. Let $\om$ be a highest weight of $\phi$.
Suppose that the maximal weight multiplicity of $\phi$ equals $2$. Then the pairs $(L,\om)$ are as in Table $1$.
\end{theo} 

The table additionally  provides the dimension of $\phi$ and the number $n_i$ of the weights of multiplicity $i$ in $\phi$ for $i=1,2$.  For readers' convenience we record in Table 2
 \ir representations of simple Lie algebras with all weight of \mult 1.

\bigskip
\centerline{Table 1: Irreducible representations of simple Lie algebras and algebraic groups}

\centerline{in characteristic $0$ with maximum weight multiplicity $2$}

\bigskip\small{
\begin{center}
\begin{tabular}{|c|c|c|c|c|}
\hline
type&highest weight&$n_1$&$n_2$&dimension\\
\hline
$A_2$ & $ (1,a),(a,1) $ & $3a+3$ & a(a+1)/2 &$(a+1)(a+3)$  \\
\hline
$A_3$ & $ (110), (011),(020),(030)$ & $12,12,18,38$ & $4,4,1,6$ &$20,20,20,50$  \\
\hline
$A_n,n>3$ & $(110...0),(0...011)$ & $n(n+1)$ & $n(n^2-1)/6$ &$n(n+1)(n+2)/3$  \\
&(020...0),(0...020) &$n^2(n+1)/2$&$n(n-2)(n-1)^2/24$ & $n(n+2) (n+1)^2/12$\\
\hline
$B_2,C_2$ & $ (11),(02),(20),(03),(30)$ & $8,12,8,12,20$ & $4,1,2,4,5$ &$16,14,10,20,30$ \\
\hline
$C_3$ & $ (010)$ & $12$ & $1$ &$14$  \\
\hline
$C_4$ & $(0001),(0010) $ & $40,32$ & $1,8$ &$42,48$  \\
\hline
$C_5$ & $(00001)$ & $112$ & $10$ &$132$\\
\hline
$F_4$ & $(0001)$ & $24$ & $1$ &$26$\\
\hline
$G_2$ & $(01)$ & $13$ & $1$ &$14$\\
\hline
\end{tabular} 
\end{center}
}

\med

To read Table 1  observe that  the fundamental weights (and the simple roots) of $L$ are ordered as in \cite{Bo} and a weight
$\om$ is written as a string of non-negative integers,  where $(0... 010... 0)$ with 1 at the $i$-th position  is the $i$-th fundamental weight. 
The fundamental weights are also denoted by $\om_1\ld \om_n$, where $n$ is the rank of $L$. 

Note that the above classification is equivalent to the classification of \ir \reps of simple algebraic groups $G$ with Lie algebra $L$ over the complex numbers with maximum weight \mult 2. 
In this paper we prefer to use group theory terminology and notation. 
A similar problem can be considered for simple algebraic groups over an algebraically closed field of prime characteristic. We plan to address this
more complex problem in a subsequent work.   

Note that in \cite{TZ18} we determined the \ir \reps of $L$ with at most one weight \mult is greater than 1. (This work also contains 
a similar result for simple algebraic groups over a field of prime characteristic.) In \cite{Seitzclass} and \cite{SZ1} is obtained a classification of 
\ir \reps of simple algebraic groups in arbitrary characteristic with all weight of \mult 1.

\med
{\it Notation} Let $\ZZ$ be the ring of integers, $\ZZ^+$ the set of non-negative integers. For a real number $r$ we denote by $[r]$ the maximal integer $k$ such that $k\leq r$.

 If otherwise is not explicitely stated, $G$ denotes a simple (non-abelian) linear algebraic group defined over an algebraically closed field of characteristic 0. Then $G$  is of type $A_n$, $B_n$, $C_n$, $D_n$, $E$, $n=6,7,8$, $F_4$ or $G_2$. For brevity, we write $G=A_n$ to say that $G$ is of type $A_n$, etc. 

We fix a maximal torus  $T$ of $G$. The set of  rational homorphisms $T\ra F^\times$ is denoted by $\Om(G)$, and the elements $\om\in\Om(G)$ are called weights of $G$. The conjugation action of $T$ on $G$ determines the $T$-weights of $G$, which are called roots (or $T$-roots)  of $G$. As maximal tori of $G$ are conjugate, the choice of a maximal torus is immaterial. The set $\Om(G)$ is a $\ZZ$-lattice of finite rank, which is called the rank $n$ of $G$. The set of roots  is denoted by $\Phi(G)$,   and  $\Phi(G)\subset \Om(G)$.  One chooses a subset  $\Pi(G)=\{\al_1\ld \al_n\}$ of $\Phi(G)$  whose elements are called simple roots.    We order simple roots as in \cite{Bo}. These determine a basis of  $\Om(G)$ which elements are called fundamental weights and denoted by $ \om_i$, $ i=1\ld n$. 
Every weight therefore is a linear combination $\sum a_i\om_i$ with $a_i\in\ZZ$. For brevity, we write weights as strings $(a_1\ld a_n)$.  Weights 
$(a_1\ld a_n)$ with non-negative $a_1\ld a_n$ are called dominant; the subset of dominant weights in $\Om(G)$ is denoted by $\Om^+(G)$.
 (We often omit $G$ if this is clear from the context.)
If $\om,\om'\in\Om(G)$ then we write $\om\succ\om'$ or $\om'\prec\om$ if $\om'=\om-\sum b_i\al_i$ for distinct $\al_i\in\Pi(G)$ and $b_i\in\ZZ^+$ for $i=1\ld n$. If $\om,\om'$ are dominant and $\om\succ\om'$ then we say that $\om'$ is a subdominant weight for $\om$. 

A semisimple  subgroup $H$ of $G$ normalized by $T$ is called  a subsystem subgroup (with respect of $T$). Then $H$ is determined by a certain subset $S$, called the subsystem base, 
of $\Phi(G)$ which forms a set of simple roots of $H$, see \cite[Section 13.1]{MTbook}.  We frequently deal with special cases of subsystem subgroups $X=X_S$ defined by a subset $S\subset \Pi(G)$. These are generated by the root subgroups $X_{\pm \al}$ with $\al\in S$. 
In these cases, by technical reason,  it  is convenient  to define $S$ as a string $(*...-...*)$ of length $|\Pi(G)|$
writing $-$ at the places that   correspond to subset $\Pi\setminus S$ of $\Pi$. For instance, if $G=C_3$ then $(-**)$ defines a subsystem subgroup of type $C_2$ and $(**-)$ defines the subsystem subgroup of type $A_2$. (Of course, to decide on the type of $X$ one has to keep in mind the Dynkin diagram and root ordering.)

 If $\lam\in \Om^+(G)$ we write $V_\lam$ for an \ir $G$-module with \hw $\lam$. If $\mu$ is a weight of  $V_\lam$ then $m_\lam(\mu)$
is the \mult of $\mu$ in  $V_\lam$. If $V$ is an \ir $G$-module then we write $V=V_\lam$ to say that $\lam$ is the \hw of $V$. If $\mu\in \Om^+(G)$ and $ \mu\prec\lam$ then $\mu$ is a weight of $V_\lam$; below we use this fact with no reference. 

We denote by $\Omega_k(G)$ the set of 
 dominant weights $\lambda $ such that $V_\lam(G)$ has no 
dominant   weight 
of multiplicity greater than $k$, and set $\Omega'_k(G)= \Omega_k(G)\setminus \Omega_{k-1}(G)$ for $k>1$.  Thus, Table 1 lists $\Om_2'(G)$.The weight multipicities of certain small representations of $G$ are given in \cite{lub} and \cite{BPM}.

\section{Preliminaries}\label{sec:omega2}
 
Throughout this section we take $G$ to be a simply connected simple algebraic group over \acf  $F$ of characteristic $p\geq 0$.

\def\hw{highest weight } 

\begin{lemma}\label{sh1}  \rm {\cite{E}, \cite[Theorem 1.3]{BZ}}   
Let  $\om =\sum_{i=1}^r a_i\om_i$ and $\mu=\sum b_i\om_i$ 
  be  dominant weights such that $\mu\prec\om$.     Suppose that $a_i\geq b_i$. Let $\nu\prec\mu$. Then $m_\om(\nu)\geq m_\mu(\nu)$.\end{lemma}

\bl{bz1}  {\rm \cite[Theorem 1.3]{BZ}} Let $\mu\prec\nu\prec\lam$ be dominant weights of $G$. Then $m_\lam(\mu)\geq m_\lam(\nu)$.\el

  \begin{lemma}\label{44}  Let  $X$ be a subsystem subgroup of $G$ normalized by 
$T$.  Let $\om\in \Omega_k(G)$. Let $V_X(\mu)$ be an $FX$-composition factor of $V_\om$, for some dominant weight $\mu$ in the character group of $X\cap T$. Then $\mu\in\Omega_k(X)$. \end{lemma}

\begin{proof} Suppose the contrary. 
Write $TX=XZ$, where  $Z=C_T(X)^\circ$. Let $0\subset M_1\subset\cdots\subset M_t=V_\om$ be an 
$F(XT)$-composition series of $V_\om$. Then there exists $i$ such that $V_X(\mu) \cong M_i/M_{i-1}$.
Now $M_i = M_{i-1}\oplus M'$ as $FT$-modules, $Z$ acts by scalars on $M'$ and
the set of $(T\cap X)$-weights in $M'$ (and their multiplicities) are precisely the same as in 
$V_X(\mu)$. 

Let $\nu$ be a  weight of $V_X(\mu)$.   
Then $\nu$ corresponds to a   $T$-weight $\nu'$ of $M'$. 
Therefore if $\nu$ is a $(T\cap X)$-weight occurring in $V_X(\mu)$ with
multiplicity greater than $k$,  then the \mult of weight $\nu'$ in $M'$ is greater than $k$. 
Then the \mult of some  dominant weight $\nu_1$ of $G $ in $M$ is greater than $ k$. This is a contradiction.  The result follows.\end{proof}

Recall that we use the simple root ordering as in \cite{Bo}. Note that $\om|_X$ means $\om|_{X\cap T}$, where $T$ is a maximal torus of $G$ defining the root system and $X\cap T$ is a maximal torus of $X$ defining the root system of $X$.




\bl{bgt1} {\rm \cite[Lemma 2.2.8]{BGT}} Let $G$ be a simple algebraic group in arbitrary characterisitic, $\Pi(G)$ the set of simple roots and let $V$ be an \ir G-module with highest weight $'om$. 
Suppose that $\mu = \om - \Sigma_{\al\in S}
 c_\al\al  $ is a dominant  weight of $V_\om$ for some proper subset $S\subset \Pi(G)$.
Then $m_V (\mu) = m_{V '} (\mu')$, where $ V ' $ is an \ir $X$-module with \hw $ \om|_X $,   $\mu' = \mu|_X$ and $X =\lan U_\pm \al |\al\in   S\ran$.\el

Recall that dominant weight $\om$ of $G$ and a weight $\nu\prec \om$, we denote by $m_\om(\mu)$ the \mult of  $nu$ in an \ir $G$-module with highest weight $\om$. 

\begin{propo}\label{ky6} {\rm  \cite[Proposition A]{CAV}} Let $G$ be a simple algebraic group of rank r,     $\om =\sum_{i=1}^r
 a_i\om_i$  a dominant  weight and
let $\mu\prec \om$ be a dominant weight such that $\mu = \om -\sum_{i=1}^r
 c_i\al_i$ with $c_1\ld c_r\in\ZZ^+$. Also, assume that $J$
is a subset  of $\{1\ld r\}$ with the property that $c_j \leq a_j$ for all $j \in J$. Set $\om' := \om-\sum_{j\in J}  (a_j-c_j)\om_j$
and $\mu' := \mu-\sum_{j\in J} (a_j-c_j)\om_j$. Then $m_\om(\mu) = m_{\om'}(\mu')$.\end{propo}

\med
Remark. In Proposition \ref{ky6} the weight $\mu'$ is not necesserily dominant.

\bl{bos9} {\rm \cite[Lemma 3.7]{BOS2}} Let G be of type $A_n$, and let $\phi$ be an \irr of G with \hw $\om=\sum_{i=j}^k a_i\om_i$ for  $1\leq j<k\leq n$ and $a_ja_k\neq 0$. 
Then $\phi$ has a weight of \mult at least $k-j$. \el

\section{Modules with maximum  weight multiplicity   $2$}

 We first consider the case of groups of type $A_2$, in this case, in contrast with the groups $A_n$ with $n>2$, there are infinite series of \ir $G$-modules with maximal weight \mult 2.

\subsection{Groups of type $A_2$}

In this section $G=A_2$. Note that $\al_1=2\om_1-\om_2, $ $\al_2=2\om_2-\om_1, $ so $\al_1+\al_2=\om_1+\om_2 $.




\bl{ca1} Let $G=A_2$. Let $\lam=(a,b)$ with $a,b>1$. Then the \mult of
weight $\lam-2(\al_1+\al_2)$ in $V_\lam$ equals $3$.
\el

\bp In Proposition \ref{ky6} take $J=\{1,2\}$ and we can take 
 $\mu=\lam-(22)$ as $\al_1+\al_2=\om_1+\om_2$.  Then $\lam'= (22) $, $\mu'=\mu-(\lam-(22))= (00)$. By \cite{lub}, the \mult of weight 0 in 
$V_{(22)}$ equals 3, and the result follows by Proposition \ref{ky6}. \enp

\bl{12w} Let $G=A_2$ and $\om=(1,b)$, $b>0$.

$(1)$  $\om\in \Om'_2(G)$. 

$( 2)$ The weights $(k+1,b-2k)=\om -k\al_2$ with $0\leq 2k\leq b$ are the only dominant weights of \mult $1$ in $V_\om$; their number equals  $[b/2]+1;$

$( 3)$ The number of   weights in $V_\om$ of  \mult
 $1$ equals $3b+3$ and those of  \mult
 $2$ equals $
b(b+1)/2$.  \el


\bp Let $\mu$ be a weight of $V_\om$. Then $\mu=\om-\beta$,   where $\be= c_1\al_1+c_2\al_2$ for some integers $c_1,c_2\geq 0$. 
Note that   $\mu$ is dominant then  either $\be= -c_2\al
_2$, $0<c_2\leq b/2$ or  $\be=c_1\al_1-c_2\al_2$, $c_1,c_2>0$. (Indeed, let $\mu=\om- c_1\al_1-c_2\al_2=(1,b)-c_1(2,-1)-c_2(-1,2)=(1-2c_1+c_2, b+c_1-2c_2)$. This weight is dominant \ii $1-2c_1+c_2,b+c_1-2c_2\geq 0, $ and hence either   $c_1=0$, or $c_1,c_2>0$.) In addition, $c_2<b$. (Indeed, otherwise $c_1\leq 1$ and $b+c_1-2c_2\geq 0$ implies 
$b+1\geq 2c_2\geq 2b$, a contradiction.)  

\med
(1) In  notation of  Proposition \ref{ky6}, and additionally set 
$\gamma= \sum_{j\in J} c_j\om_j$.
Choose $J=\{2\}$ in  Proposition \ref{ky6} and let $\mu=\om-c_1\al_1-c_2\al_2$ be a dominant weight of $V_\om$. 
Then $\gamma=(0,c_2)$ and $\lam' =(1,b)-(0,b-c_2)=(1,c_2)$. The cases with $b\leq 2$ follows from \cite{lub}. We use induction on $b$,  the base of induction is $\om\in\{(1,1),(1,2)\}$. 
By induction assumption, we have $m_{\om'}(\mu') \leq 2$ for every weight
$\mu'$ of $V_{\om'}$. By  Proposition \ref{ky6}, $m_\om(\mu)=m_{\om'}(\mu') $,
whence $m_\om(\mu)\leq 2$.  This yields the first assertion of the lemma. 
\med

(2) 
 By  \cite[Theorem 1.1]{BZ}, $m_\lam(\mu)=1$ if $c_1=0$ and $m_\lam(\mu)>1$ otherwise. So the claim follows from $(1)$.

\med

(3)  Let $n_1,n_2$ be as in Table 1. By the above, the dominant weights of \mult $1$ are of the form $\om -k\al_2=(1,b)-k(-1,2)=(-k+1,b-2k)=(b-k+1)\ep_1+(b-2k)\ep_2)$ with $0\leq 2k\leq b$. The other 
weights of \mult $1$ are in the $W$-orbits of dominant ones. As $2k\leq b$, the coefficient of $\ep_1$ is non-zero and differs from that of $\ep_2$. The latter is non-zero unless $b=2k$.  Then  the stabilizer of   such a weight in $W$ is trivial, unless $b=2k$ when this   of order 2. As $|W|=6$, the $W$-orbit of  $\om -k\al_2 $ is of size 6, unless $b=2k$ when this is of size 3.
Let $d_1$ be the number of dominant weights of \mult 1.  

If $b$ is odd then every orbit is of size 6 and hence $n_1=6d_1=6(b+1)/2=3(b+1)$. If $b$ is even then $n_1=3+6(d_1-1)=6(b/2)+3=3(b+1)$.  
 
 As the weights of $V_\om$ are of \mult 1 or 2, it follows that $n_2= (\dim V_\om-d_1)/2$. The dimension of $V_\om=(b+1)(b+3)$; this follows by a formula (\ref{e5})
 for $\dim V$ recorded in Section 4, see Lemma \ref{dd1} below.  So  $n_2=((b+1)(b+3)-3(b+1))/2=b(b+1)/2 $.  \enp


\begin{propo}\label{2aa} Let $G=A_2$. Then  $\Om_2'(G)=\{(1,k),(k,1)\}$ for $k\geq 1$. \end{propo}

\bp This follows from Lemmas \ref{ca1},  \ref{12w} and Table 2.\enp
 
\subsection{Groups of type $A_n, n>2$}

The \f lemma is proved in \cite[Lemma 3.2.3]{KT}:  

\bl{a33} Let $G=A_3$. Then  $  \Om'_2 (G) =\{(110),(011),(020),(030)\}$.\el

\bl{su1} Let $G=A_n,n>3$ and let $\om$ be as in Table $1$. Then $\om\in\Om_2'(G)$.
\el

\bp 
 As the sets of weight mutiplicities of dual modules are the same, it suffices to consider the entries $(110...0)$ and $(020...0)$ of the table. 

Note that subdominant weights of $2\om_2$ are $\om_1+\om_3$ and $\om_4$,
and $\om_3$ is the only subdominant weight of $\om_1+\om_2$. Moreover, $\om_4=\om_1+\om_3-\al_1-\al_2-\al_3$, $\om_1+\om_3=2\om_2-\al_2$, 
$\om_3=\om_1+\om_2-\al_1-\al_2$. We use Lemma \ref{bgt1} with $S=\{\al_1,\al_2,\al_3\}$. Then $X$ is of type $A_3$, $ \om|_X =(020),(110)$
if $\om=2\om_2,\om_1+\om_2$, respectively.
In addition,  $\mu|_X=(000),(101),(001)$ if $\mu=\om_4,\om_1+\om_3,\om_3$, respectively.  By the above, the weight mmultiplicities of \ir $A_3$-modules 
with \hw $ (020)$ or $(110)$ are at most 2. So this is the case for $V_\om$ by Lemma  \ref{bgt1}.\enp


\bl{a4b}  Let $G=A_4$. Then     $\Om'_2(G)=\{ (1100),(0011)$, $(0020),(1100),(0200)\}$.\el

\bp Recall that  $\Om_1'(G)=\{k000),(0100),(0010),(000k)\}$ for $k>0$ and  $\Om_1(A_3)=\{(k00),(010)$, $(00k) \}$ for $k>0$. Let $\lam=(a,b,c,d)\in \Om_2'(G)$. By Lemma \ref{bos9}, $ad=0$. We can assume that  $a\leq d$, and hence $a=0$.

By Lemmas \ref{44} and \ref{a33}, $(a,b,c), (b,c,d)\in \Om_2(A_3)=\Om_1(A_3)\cup \{(110),(011), (020),(030)\}$. If $(a,b,c)\in \{(020),(030)\}$ then $(b,c,d)=(20d)$ or $(30d)$, respectively, and $d=0$ by Lemma \ref{a33}. So $(a,b,c,d)\in \{(0200), (0300)\}$, respectively. The former case is recorded in the statement, together with $(0020)$, and
the case $(0300)$ is ruled out by \cite{lub}. Suppose that  $(a,b,c) =(011)$. 
Then  $(b,c,d)=(11d) $, whence $d=0$  by Lemma \ref{a33}, so   
$\lam=(0110)$. This case  is ruled out by \cite{lub}. Finally, if $(a,b,c)=(001)$ or $(010)$ then $(bcd)=(00d)$ or $(01d)$.  In the former case 
$\lam=(000d)\in\Om_1(G)$.      In the latter case $d=0,1$ as $(01d)\notin\Om'_2(A_3)$ for $d>1$. So $\lam=(0010)$ or $(0011)$, as required. \enp

Remark. The case with $a=d,b=c$ of Lemma \ref{a4b} is settled in \cite[Lemma 3.2.5]{KT}.

\bl{gc1} Let $G=A_n$, $n> 4$. Then $\Om_2'(G)=\{(110\ldots 0), (0\ldots 011), (020\ldots 0)$, $(0\ldots 020)\}$.\el

\bp  Let $\om=(a_1\ld a_n)$.  By Lemma \ref{bos9}, $ad=0$. We can assume that  $a_1\geq a_n$, and hence $a_n=0$.  

Let $X=A_{n-1}$
be the subgroup of $G$ generated by the root subgroups $X_{\pm \al_i}: i=1\ld n-1$. By Lemma \ref{44}, $(a_1\ld a_{n-1})\in\Om_2(A_{n-1})$. We use induction on $n\geq 4$. The base of induction  $n=4$ is settled in Lemma \ref{a4b}. By induction assumption,
$(a_1\ld a_{n-1})\in \Om_1(A_{n-1})\cup (110...0),(020...0),(0...020),(0....011)$. As $a_n=0$, this implies that either $\om\in\Om_2(A_n)$ 
or $\om=(0...0200)$. This is ruled out by  applying Lemma \ref{44} to the subsystem subgroup $X$ of $G$ with subsystem base   $  \al_2\ld \al_n$.\enp

The entries in  Columns 3,4,5 of  Table 1 will be justified in Section 5.

\subsection{Type $C_n$}

 Using the tables in \cite{lub} one has to keep in mind that the root ordering in   \cite{lub}  is opposite to that in \cite{Bo} 
used below.
 




$(5)$ Suppose that $\om\neq \om_i$ for $1\leq i\leq n$. If $\om_1\prec\om$ then $\om_1+\om_2\preceq \om$; if $0\prec\om$ then $2\om_1\preceq\om.$
%

\bl{lu4} Let $G=C_2 $ and let  $(a,b)\in\Om^+(G)$.  

$(1)$ $\al_1=2\om_1-\om_{2}$, $\al_2=2\om_2-2\om_{1}$, $\al_1+\al_2=\om_2$ and $(a,b) \succ (a,b-1)$

$(2)$ $\{(20)$, $(02)$, $(11)$, $(03) ,(30)\}\in \Om'_2(C_2)$. 

$(3)$ $(40),(50), (21),(31), (04)\not\in\Om_2(C_2)$.  \el

\bp (1) is well known, and (2),(3) follows from \cite{lub}. \enp

\bl{bc2} $\Om'_2(C_2)=\{(2,0)$, $(0,2)$, $(1,1)$, $(0,3) ,(3,0)\}$. \el

\bp Let $\lam=(a,b)\notin\Om_1(C_2)$ be a dominant weight, and $(a,b)$ is not 
a weight from Lemma \ref{lu4}(2). Suppose first  that $b=0$. Then $(a,0)\succ (a-2,1)\succ  
(a-2,0)$ by Lemma \ref{lu4}(1).  
By Lemma  \ref{sh1}, if $(a-2,0)\notin \Om_2(C_2)$ then so is $(a,0)$.
By Lemma \ref{lu4}(3), $(a,0)\notin\Om_2$ for $a=4,5$, whence the result for $b=0$.

Let $b>0$. Then $(a,b)\succ (a,b-1)\succ (a,0)$. By  Lemma  \ref{sh1}, it suffices to prove the result for $a\leq 3$.

Suppose that $a=0, b>3$. Then $(0,4)\notin \Om_2(C_2)$ Lemma \ref{lu4}(3), whence $(0,b)\notin \Om_2 (C_2) $ for $b>4$ by  Lemma  \ref{sh1}.  In addition,  by Theorem \ref{sh1}, we have the following conclusions.

If $a=1, b>1$ then $(12)\notin \Om_2(C_2)$, whence $(1,b)\notin \Om_2(C_2)  $ for $b>2$.  

If $a=2$, $ b>0$ then $(21)\notin \Om_2(C_2)$, whence $(a,b)\notin \Om_2(C_2)  $ for $b>1$.   

If $a=3$, $b>1$  then $(31)\notin \Om_2(C_2)$, whence $(3,b)\notin \Om_2 (C_2) $ for $b>1$   \enp

\bl{lu2} Let $G=C_3$ and $\lam\in \{(000),(100),(001),(010),(101)\}$.  Then $\lam+2\om_i\notin \Om_2(C_3)$ for $i=1,3$ and if $\lam\neq (000)$ then $\lam+\om_2\notin \Om_2(C_3)$.
\el

\bp One observes that the weights $\lam+2\om_i$, $i=1,3$, and  $\lam+2\om_2$ are listed in the $C_3$-table of \cite{lub}, 
so the result follows by inspection of the table. 
\enp

\bl{c33} Let $G=C_3$. Then $\Om_2'(G)=\{(010)\}$.\el

\bp
Let $\om=\sum a_i\om_i\in\Om_2'(C_3)$ be a dominant weight of $G$. If $a_i\leq 1$ for $i=1,2,3$ then the result follows by inspection in \cite{lub}. 
Suppose that  $a_i>1$ for some $i\in\{1,2,3\}$. Observe that $\om_2$ and  $2\om_j$, $i=1,3$,  are positive roots, see \cite[Table III]{Bo}. 
Therefore, if $a_i>2$ for $i\in\{1,3\}$ then $\om\succ \om-2\om_i$ is a dominant weight, if $a_2>2$ then $\om-\om_2$  is a dominant weight of $V_\om$. \itf $\om\succ\om'$, where $\om'=\sum a_i'\om_i$ is a dominant weight of $V_\om$ for some $a_1',a_2',a_3'\leq 1$
and if $a_2>0$ then we can assume $a_2'=1$. By Lemma \ref{sh1}, we have $\om'\in\Om_2(C_3)$. Then $\om'\in\{(000),(100),(001),(010)\}$
by \cite{lub}.   If $\om$ is not in this set then $\om'+\nu\preceq \om$ is  a weight of $V_\om $ for some  $\nu\in\{2\om_1,\om_2,2\om_3\}$, and $\om'+\nu\in \Om_2(G)$ by   Lemma \ref{sh1}. This contradicts Lemma \ref{lu2},  unless $\om'=0, \nu=\om_2$. If $\om=\om_2$, we are done. Suppose that $\om\neq\om_2$. Let $\nu'$ be a minimal dominant weight such that $\nu'\succ \om_2$. Then  $\nu'\in\{(210),(020),(012 )\}$.   As $\om\neq \om_2$, we have   
  $\om\succeq \nu'$ for some $\nu'\in\{(210),(020),(012 )\}$.  These cases  
are ruled out by \cite{lub} and   Lemma \ref{sh1}. \enp

\bl{luc4}  {\rm \cite{lub}} $(1100),(0100),(0010),(0001),(1001),(2000), (3000)\notin \Om_2(C_4)$.\el

\bl{c44} Let $G=C_4$. Then $\Om_2'(G)=\{(0001),(0010)\}$.\el

\bp By \cite{lub}, we have $(1100),(0100),(0010),(0001),(1001),(2000), (3000)\notin \Om_2(C_4)$.

 Let   $\lam=(a,b,c,d)\in \Om'_2(C_4)$. We apply Lemma \ref{bgt1} to subsystem subgroups $(-***)$ of type $C_3$ and $(***-)$ of type $A_3$.
By Lemma \ref{c33} and Table 2 we have $(abcd)\in (a000),(a100),(a001),(a010)$; by Lemma \ref{a33}   and Table 2 $(abc)\in\{
(a00),(00c), (010),(110)$, $(011)$, $(020)\}$. So  $(abcd) \in \{(a000), (1100),(0100),(0010),(0001)$, $(1001)\}$. If  $\om=a\om_1$
then $a>1$ by Table 2. As $\om_1$ is a simple root,  we have  $\om=a\om_1\succ (a-2)\om_1\succ ... \succ\nu$ for $\nu=2\om_1,3\om_1$. By   Lemma    \ref{sh1}, $\nu\in\Om_2(C_4)$, which contradicts Lemma \ref{luc4}.\enp

\bl{c4n} Let $G=C_n,n\geq 5$. Then $\Om_2'(G)=\emptyset$.  \el

\bp We use induction on $n\geq 4$. For  $n=4$ see Lemma \ref{c44}. Let $\om=\sum a_i\om_i\in\Om_2(G)$. 
We apply Lemma \ref{bgt1} to subsystem subgroups $(-*...*)$ of type $C_{n-1}$ and $(*...*-)$ of type $A_{n-1}$.

Observe first that $a\om_1$ for $a>1$ and $ a\om_1+\om_n$ for $a>0$ are not in $    \Om_2(G)$, $n>4$. Indeed, consider the $C_3$ subsystem subgroup $X$ with system base
$\al_1'=\{\al_1+...+\al_{n-2},\al_2'=\al_{n-1},\al_3'=\al_n\}$. By Lemma \ref{44}, $V_\om|_X$ contains an \ir constituent with \hw  $a\om_1|_{T\cap X}=a_1\om'_1$ and $a_1\om'_1+\om_3'$, respectively, where $\om_i', i=1,2,3,$ are the corresponding fundamental weights of $X$.  This contradicts Lemma \ref{c33} and Table 2.

In addition, $\om_1+\om_2\notin\Om_2(C_n)$. Indeed,  take for $X$  a subsystem  subgroup  of type $C_4$ with  system base
$\{\al_1'=\al_1,\al_2'=\al_2+...+\al_{n-2},\al_3'=\al_{n-1},\al_4'=\al_n\}$. By Lemma \ref{44}, $V_\om|_X$ contains an \ir constituent $W\in\Om_2(C_4)$, whose  \hw  of $W$ is $\om |_{T\cap X}=\om'_1+\om_2'$, where $\om_i', i=1,2,3,4,$ are the   fundamental weights of $X$.   As $(1100)\notin \Om_2(C_4)$, we get a contradiction.

Suppose first that $n=5$. Then $\om\in\{(a_10000),(a_11000),(a_10001),(a_10100) \}$ by Lemma \ref{c44} and Table 2. Using the subsystem subgroup $(****-)$ of type $A_4$,  Lemma \ref{a4b} and Table 2,
we are left with $\om\in\{(a_10000), (11000),(01000)$, $(00010),(a_10001)\}$. The first  case is ruled out above, and the next two cases  are ruled out by \cite{lub}.

 Consider $\om=(a_10001)= a\om_1+\om_5$. If $a$ is odd then $\om\succ \om_1+\om_5$. As $ \om_1+\om_5\notin\Om_2(G)$ by \cite{lub}, the result   follows by  Lemma  \ref{sh1}. Let  $a_1$ be even. If $a_1=0$ then the result follows by \cite{lub}. Let $a_1>0$.  Then  $\om\succ 2\om_1+\om_5$.
 By \cite[p. 228]{BPM}, $\om=(20001)\notin\Om_2(C_5)$, and the the result again follows by  Lemma \ref{sh1}. 
 
Let $n=6$. As above we have $\om=(a_100000),(a_1 10000),(a_100001)$ by induction and  Lemma \ref{bgt1}.    Applying Lemma \ref{gc1},
we are left with $\om=(a_100000),      (110000),(a_100001)$.  These cases  are sorted out above.  

Let  $n>6$. Then $\om=(a_100000) )$ by induction and  Lemma \ref{bgt1}.  This case is rules out above. \enp

\subsection{Groups of types  type $B_n,D_n$} In this section 
 we show that $\Om_2(G)=\Om_1(G)$ for $G$ of type $B_n.n>2$ and $D_n,n>3$.

\bl{dd4}{\rm \cite[Lemma 3.2.4]{KT}}  $\Om_2(D_4)=\Om_1(D_4)$, that is, $\Om_2'(D_4)$ is empty.
 \el

\bl{23b} Table $1$ is correct for $G=B_n,n>1$ or $D_n,n>3$.\el

\bp Note that $B_2\cong C_2$ and the weights $(a,b)$ for $B_2$ are interpreted as $(b,a)$ for $C_2$. The entries in Table 1 for $C_2$ do not change    under replacing $(ab)$ by $(ba)$, so we conclude that the entries for $B_2$ coincide with those for $C_2$.

Let $G=B_3$.  The cases with $a_i\leq 2 $ for $i=1,2,3$  are treated in \cite{lub}, and the result follows by inspection. 

Note that $\om_1=\al_1+\al_2+\al_3,\om_2=\al_1+2\al_2+2\al_3$ are positive roots, and $2\om_3=\al_1+2\al_2+3 \al_3.$
Let $\om=\sum a_i\om_i$. Then $\om\succ \om'=\sum a_i'\om_i$ for some $a'_1,a'_2,a'_3\leq 1$. By Lemma \ref{sh1}, $\om'\in\Om_2(B_3)$, whence 
$\om'\in \{(000),(100),(001)\}$ by the above. If $\om\neq \om'$, then there exists a dominant weight $\nu=\sum b_i\om_i $ such that $\om'\prec\nu\preceq \om$ and $\nu-\om'\in\{\om_1,\om_2,2\om_3\}$. Then $b_1b_2b_3=0$ and $b_1,b_2\leq 2, b_3\leq 3$. All these cases are recorded in \cite{lub}, and one   checks that  $\nu\notin \Om_2(B_3) $. This contradicts   Lemma \ref{sh1}.

Next we examine the case where  $G=D_4$. Let $\om=\sum a_i\om_i\in \Om_2(D_4)$.  By inspection in \cite{lub}, one concludes that if
 $a_1,a_2,a_3,a_4\leq 1$ and  $\om\in\Om_2(D_4)$ then $\om\in\Om_1(D_4)$. 
One observes that $\om_2,2\om_i$, $i=1,3,4$, is a sum of positive roots. \itf  $\om\succ \om'=\sum a_i'\om_i$, where $a_1,a_2,a_3,a_4\leq 1$.
Then  $\om'\in\Om_2(D_4)$ by  Lemma \ref{sh1}, and hence $\om'\in\Om_1(D_4)$ by the above. Recall that $\Om_1(D_4)=\{\om_1,\om_3,\om_4\}$.
 If $\om\neq \om'$, then there exists a dominant weight $\nu=\sum b_i\om_i $ such that $\om'\prec\nu\preceq \om$ and $\nu-\om'\in\{2\om_1,\om_2,2\om_3,2\om_4\}$. Such weights $\nu$ are recorded in \cite{lub}, and one observes that $\nu\notin \Om_2(D_4)$. 
This contradicts   Lemma \ref{sh1}.

Let $G=B_n,n>3$ or $D_n,n>4$. We use induction on $n$ with the base of induction to be $n=3$, 4,  respectively. Let $\om=\sum a_i\om_i\in\Om_2(G)$. Applying Lemma \ref{bgt1} to a subsystem subgroup 
with $S=\{\al_2\ld\al_n\}$ and using the induction assumption, we conclude that $\om\in \{(a_1,0...0),(a_1,10...0)\}$ in case $B_n$ and  
$\om\in\{(a_1,0...0),(a_1,10...0), (a_1,0...010),(a_1,0...01)\}$  in case of $D_n$. Applying   Lemma \ref{bgt1} to a subsystem subgroup with $S=\{\al_1\ld\al_{n-1}\}$  in case $B_n$,
we are left with $\om=(a_1,0...0),(110...0)$, $(010...0)$. In case of $D_n$ we choose $S=\{\al_1\ld\al_{n-1}\}$ if $\om=(a_1,0...010)$ and $S=\{\al_1\ld\al_{n-2},\al_n\}$ if $\om=(a_1,0...01)$. Then we are again left with  $\om=\{(a_1,0...0),(110...0),(010...0)\}$. 

The case  with $\om=\om_2$  both for $B_n,D_n$ is ruled out by \cite[Table 2]{TZ18}, in fact, this is well known.  Observe that $2\om_1$ is not in $\Om_2(G)$ for both  $B_n,D_n$, see for instance \cite[Table 2]{TZ18}. Next we mimic the reasoning in \cite[Lemmas 6.14 and  6.15] {Seitzclass}.

Let $G=D_n$. Then the case $\om=\om_1+\om_2$ is ruled out by applying  Lemma \ref{44} to a subsystem subgroup $X$ with  system base   $\al_1, \al_2+...+\al_{n-2},\al_{n-1}-\al_n,\al_{n-1}+\al_n$.  Note that $X$ is of type $D_4$ and,  by Lemma \ref{44},  $V|_X$ contains a composition factor with \hw $\om|_X$. As $\lan \om_1+\om_2,  \al_2+...+\al_{n-2}\ran=\lan \om_1+\om_2,  \al_2\ran=1$, we have   
$\om|_{X}=(1100)$.  By \cite{lub},  this factor has a weight of  \mult  6. So $\om\notin \Om_2(D_n)$. 
Similarly, if $\om=3\om_1$ then $V|_X$ contains a composition factor with \hw $\om|_X$, where $\om=(3000)$. This factor has a weight of \mult 3, and hence $\om\notin \Om_2(D_n)$ by Lemma \ref{44}. 

Suppose that $\om=a\om_1$, $a>3$. Then $\om\succ b\om_1$ for $b\in \{2,3\}$. By the above, $b\om_1\notin \Om_2(D_n)$ and so is $a\om$ by Lemma \ref{sh1}.
 
Next suppose that $G=B_n, n>3$. Define a subsystem subgroup $X$ by taking  $S=\{ \al_1\ld\al_{n-1}, \al_{n-1}+2\al_n\}$ as a set of simple roots of $X$. Then $X$ is of type $D_n$ and $\om':=\om|_X=(a0...0)$ or $(110...0)$ if $\om=a\om_1$, $\om_1+\om_2$, respectively. By Lemma \ref{44}, $\om|_X\in \Om_2(D_n)$. This contradicts the result for $D_n$ already proved, unless $a=1$. This implies the result. \enp

\subsection{Exceptional types }

\bl{e22} Table $1$ is correct for the simple   groups of exceptional types.\el

\bp
Let $G=E_n$, $n\in\{6,7,8\}.$ Let $\om=\sum a_i\om_i\in\Om_2(G)$. Applying Lemma \ref{44} to the subsystem subgroup $X=(-*...*)$ of type $D_{n-1}$ 
we get $\om|_X\in \Om_2(D_{n-1})$, so $\om\in\{a_1\om_1, a_1\om_1+\om_j, j=2,3,n \}$ by Tables 1,2.

Let $n=6$.  then we
repeat this using the subsystem subgroup $X=(*...*-)$. 
Then  $\om\in\{a_6\om_6, a_6\om_1+\om_j, j=1,2,5 \}$ by Tables 1,2. The only common weights of these two lists are $0,\om_1,\om_6,\om_2,\om_1+\om_6$. As $0,\om_1,\om_6\in\Om_1(E_6)$ and $\om_2\notin \Om_2(E_6)$ by \cite{lub}, we are  left with $\om=\om_1+\om_6$. This case is ruled out by applying Lemma \ref{44}  to  the subsystem subgroup  $X=(*-*...*)$ of type $A_5$. (The  weight $\om|_X$
is the highest weight of the adjoint \rep of $A_5$, in which the weight 0 occurs with\mult 5.)  
 
Let $G=E_7$. Similarly, using the subsystem subgroup of type $E_6$ we get $\om\in  \{a_7\om_7,\om_1+a_7\om_7,\om_6+a_7\om_7\}$. Then, by the above,  using   $X=(-*...*)$  of type $D_6$,
we conclude that  $\om\in\{0,\om_7\}$. As $0,\om_7\in\Om_1(E_7)$ by Table 2, the result follows. 
 
Let $G=E_8$. Using $X=E_7$ we get $\om\in\{ a\om_8, \om_7+a\om_8\}$. As above, take $X=(-*...*)$. Then we have $\om\in \{0,\om_8\}$. However, $\om\neq \om_8$ by \cite[Table 2]{TZ18}.

Let $G=F_4$. Then we use   the subsystem subgroup $X=(***-)$  of type $B_3$  
and the subsystem subgroup $Y=(-***)$ of type $C_3. $   
As above, Lemma \ref{44} implies $\om\in \{ a_4\om_4,\om_1+a_4\om_4,\om_3+a_4\om_4\}$ and 
$\om\in \{a_1\om_1,a_1\om_1+\om_i,i=2,3,4\}$, respectively. The common weights are $0,\om_1,\om_1+\om_4, \om_4,\om_1+\om_4$. By inspection in \cite{lub}, we get  $\Om_2'(F_4)=\om_4$. 

Let $G=G_2$. Note that $\om_1\in \Om_1(G_2)$ and $\om_2\in \Om'_2(G_2)$. In addition, $2\om_1,2\om_2,\om_1+\om_2\notin \Om_2(G_2)$ by \cite{lub}.
 Let $\om=a_1\om_1+a_2\om_2\in\Om_2'(G_2)$. Note that $\om_1,\om_2$ are positive roots. Therefore, $(a_1-1)\om_1+a_2\om_2\prec \om$ and 
 $a_1\om_1+(a_2-1)\om_2\prec \om$.  \itf that $\om\succeq \om'$, where $\om'\in\{2\om_1,2\om_2,\om_1+\om_2\}$. By Lemma \ref{sh1},
$\om'\in\Om_2(G_2)$, a contradiction. \enp

\section{The number of dominant weights of multiplicity 2}

The data in columns 3,4,5 for simple groups $G\neq A_n$ are taken from \cite{lub}. So we assume here that $G$ is of type $A_n, n>1$.  

We first compute the entries of the 5-th column of Table 1, that is,   the dimension of the reresentations listed in the table. 
As above, denote by  $V_\lam$ the \irr of a simple algebraic group $G$ over the complex number with highest weight $\lam$. Recall that $\Phi(G)$ is the set of the roots of $G$ and $\Phi^+(G)$ is that of positive roots. Let $n$ be the rank of $G$.

Weyl's dimension formula states $\dim V_\lam=\Pi_{\al\in\Phi^+(G)} \frac{\lan \lam+\rho,\al\ran}{\lan \rho,\al\ran}$, where $\rho=\frac{1}{2}\sum _{\al\in \phi^+(G)}\al$ is the half of the positive root sum. Note that the expressions of $\rho$ is terms of simple roots are provided in 
\cite[Tables  I--IX]{Bo}.

  Let $G=A_n$.  Exress $\lam=\lam_1\ep_1+...+\lam_n\ep_n$, where $\ep_1\ld  \ep_{n+1}$ are so called "Bourbaki weights", see 
\cite[Table  I]{Bo}. The positive roots of $G$ are $\ep_i-\ep_j$ $1\leq i<j\leq n+1$, $\rho=\om_1+...+\om_n=n\ep_1+(n-1)\ep_2+...+\ep_n$, whence 

\begin{equation}\label{e5}\dim V_\lam=\Pi_{1\leq i<j\leq n+1}\frac{\lam_i-\lam_j+j-i}{j-i}\end{equation}
see \cite[10.13, p.48]{Se}.

\bl{dd1} $(1)$Let $G=A_2$. The dimension of an \irr of  with \hw $a\om_1+b\om_2$ is $(a+1)(b+1)(a+b+2)/2$. In particular, if $a=1$ then this equals $(b+1)(b+3)$.

$(2)$ Let $G=A_n$. The dimensions of  \ir \reps   with \hw $\om\in\{2\om_2, \om_1+\om_2\}$ equal
 $n(n+1)^{2}(n+2)/12$, $n(n+1)(n+2)/12$, respectively.
 \el

\bp
(1) Let $G$ be of type $A_2$ and $\lam=a\om_1+b\om_2=(a+b)\ep_1+b\ep_2$. Then $\lam_1=a+b,\lam_2=b, \lam_3=0$.
Then  

$$\dim V_{(a,b)}=\frac{(a+1)(a+b+2)(b+1)}{2}.$$

\med
(2) Let $G=A_n$, $\om=2\om_2$. Then $\om=2\ep_1+2\ep_2$, so $\lam_1=2$, $\lam_2=2$, $ \lam_3\ld \lam_{n+1}=0$.
 
$$\dim V_{\om}=\frac{4.5...(n+2).3...n+1)}{  n!(n-1)! }=\frac{(n+2)(n+1)^2n}{2.3.2}=\frac{n(n+2)(n+1)^2}{12}.$$

Let $G=A_n$, $\om=\om_1+\om_2$. Then $\om=2\ep_1+\ep_2$, so $\lam_1=2,\lam_2=1, \lam_3\ld \lam_{n+1}=0$.

$$\dim V_{\om}=\frac{2.4.5...(n+2).2.3...n}{  n!(n-1)! }=\frac{(n+2)(n+1)n}{3}.$$

So the $A_n$-entries in Table 1 are correct. \enp

\med
Suppose first that $n>3$. To compute the number of weights of \mult 2 observe that  $\om_3$ is the only subdominant weight of $V_{\om_1+\om_2}$, hence it has \mult 2.
The other weights of \mult 2 are in the $W$-orbit of $\om_3=\ep_1+\ep_2+\ep_3$. As $W\cong S_{n+1}$ acts on $\ep_1\ld \ep_{n+1}$
by permutations, the stabilizer of $\om_3$ in $S_{n+1}$ is of order $6(n-2)!$, and hence the number of weights of \mult 2 is $n(n^2-1)/6$.

The dominant weights of $V_{2\om_2}$  are $\om_1+\om_3=\om-\al_2$ and $\om_4=\om-\al_1-2\al_2-\al_3$ (for $n>3$). 
 
If $n=4$ then, by Lemma \ref{bgt1}, the multiplicities of these weights equal to those for $n=4$. By \cite{lub},  the \mult of  $\om_1+\om_3$ equals 1 and that of $\om_4$ equals 2. 
As  $\om_4=\ep_1+\ep_2+\ep_3+\ep_4$, the number of weights of \mult 2 equals $(n+1)!/4!(n-3)!=n(n-2) (n^2-1)/4!$. 
 The number of the \mult 1 weights is $\dim V_{2\om_2}-\frac{n(n-2) (n^2-1)}{12}=\frac{n(n+2)(n+1)^2}{12}-\frac{n(n-2) (n^2-1)}{12}=\frac{n(n+1)[ (n+2)(n+1)-(n-2)(n-1)}{12} =\frac{n(n+1)n}{2}$.

For other use we compute $\dim V_{\om_1+\om_3}$. Here $\om_1+\om_3=2\ep_1+\ep_2+\ep_3$, so $\lam_1=2,\lam_2= \lam_3=1$, $\lam_4\ld \lam_{n+1}=0$. By  (\ref{e5}), we have 

$$\frac{2.3.5...(n+2).3.4...n.2...(n-1)}{n!(n-1)!(n-2)!}=\frac{(n+2)!n!(n-1)!}{4.2.n!(n-1)!(n-2)!}=(n^2-1)n(n+2)/8.$$


\vspace{10pt}


\newpage 
\centerline{Table 2: Non-trivial \ir \reps of simple algebraic groups} \centerline{ with all weights of \mult 1}

\bigskip\small{
\begin{center}
\begin{tabular}{|l|c|}
\hline
$~~~~~~~~~~{\rm type}$&$ \Omega_1(G)\setminus \{0\}$ \\
\hline
$A_1$&$a\omega_1, a>0$\\
\hline
$A_n, n>1$&$a\omega_1, b\omega_n,  a,b>0$\\
&$ \omega_i,\  1<i< n$\\
\hline
$B_n, n>2 $&$\omega_1,\ \omega_n $\\
\hline
$C_2 $&$\omega_1,\ \omega_2$\\
 \hline
$C_3$&$\omega_3$\\
\hline
$C_n,n>2 $&$\omega_1,$\\
 \hline
$D_n, n>3$&$\omega_1,\ \omega_{n-1},\  \omega_n$\\
 \hline
$E_6$&$\omega_1, \ \omega_6$\\
\hline
$E_7$&$\omega_7  $\\
\hline
$F_4, p=3$&$\omega_4$ \\
\hline
$G_2, p\neq 3$&$\omega_1$\\
\hline 
$G_2, p= 3$&$\omega_1,\ \omega_2$\\
\hline
\end{tabular}
\end{center}

\end{document}

 ADITIONAL MATERIAL 1

Table 2: Non-trivial $p$-restricted $G$-modules 
 whose weights are one-dimensional
\bigskip

\begin{table}[h]
\begin{tabular}{|l|c|c|c| }

        \hline
         $\,\,\,\,\,\,\,\,\,\,\, G$&highest weight 
&
weight 0 \mult&dimension\\
        \hline
          $A_n,\,\,n>1,\,\,\,\,p\not|\, (n+1)$&$\om_1+\om_n$& $n$&$n^2+2n$
\cr $(n,p)\neq (2,3),\, p\,|\,(n+1)$&$\om_1+\om_n$& $n-1$ &$n^2+2n-1$\cr

 $A_3$,  $p>3$&  $ 2\om_2$& 2  & $20$ \cr \hline

          $B_n, n>2,p\neq 2$& $\om_2$&  $n$& $2n^2+n$\cr
\,\,\,\,\, $p\,\,|\,(2n+1)$&$2\om_1$&$n$ &$2n^2+3n-1 $ \cr
\,\,\,\,\, $p\not|\,(2n+1)$&$2\om_1$& $n+1$&$2n^2+3n$\cr
 $B_2$, $ p\neq 2$& $ 2\om_2 $&  $2$&10 \cr
 \,\,\,\,\, $ p\neq 2,5$& $ 2\om_1 $&  $2$&14 \cr
\hline



$C_n$, $n>2,$&$2\om_1$ &$n$&$2n^2+n$\cr

\,\,\,\,\, $(n,p)\neq(3,3)$&$\om_2$& $n-1$ if $p\not|n$&$2n^2-n-1$\cr

&& $n-2$ if $p|n$,\,\,&$2n^2-n-2$\cr

$C_2$, $p\neq 2$&$2\om_1$&$2$&10\cr

\,\,\,\,\, $p\neq 2,5$&$2\om_2$&$2$&14 \cr

$C_4$, $p\neq 2,3$&$\om_4$&$2$&36\cr

\hline

 $D_n$, $n>3,p\neq 2$&$2\om_1$& $ n-2$ if $p|n$, & $2n^2+n-2$ if $p|n$,\\
&&$n-1$ if
$p\not|n$&  $2n^2+n-1$ if $p\not|n$\cr
 & $\om_2$& $n$&$2n^2-n-1$\cr
\,\,\,\,\, \,\,\,\,\, $p=2$& $\om_2$&
$n-(2,n)$ &$2n^2-n-(2,n)$\cr

\hline

          $E_6,$ \,\,\,\,$p=3$&$\om_2$&5  & 77  \cr
\,\,\,\,\, \,\,\,\,\,\,\,\,$p\neq 3$&& 6 &78\cr
\hline $E_7$, \,\,\,\,$p\neq 2$&$\om_1$ &7&133 \cr
\,\,\,\,\, \,\,\,\,\,\,\,$p= 2$&&6& 132 \\
 \hline
 $E_8$&$\om_8$ & 8&248\cr
\hline
          $F_4,$\,\,\,\,\,\,$p=2$&$\om_1$ &2 &26    \cr
\,\,\,\,\, \,\,\,\,\,\,\,$p\neq 2$&&4&52\cr
\,\,\,\,\, \,\,\,\,\,  $p\neq 3$& $\om_4$ &2 &26\cr
\hline
          $ G_2$, $\,\,\,p\neq 3$&$\om_2$&2&14
             \\
\hline
    \end{tabular}

\end{center} }
\end{table}


\begin{table}\label{om1}
Table 2: 
\vskip1cm

\begin{center}
\small{

\bigskip
\begin{tabular}{|l|c|}

        \hline
         $\,\,\,\,\,\,\,\,\,\,\,  G$& highest weight  
\\

  \hline
$A_1$&$a\om_1$, $\,\,\,\, 1\leq a<p$\cr
        \hline
          $A_n, n>1$&$a\om_1$, $\,\,b\om_n$, $\,\, 0\leq a,b<p$, $\,\,\om_i$, $1<i< n$
\cr &$c\om_i+(p-1-c)\om_{i+1}$,  $\,\,i=1\ld n-1,$ $\,\, 0\leq c<p$ \cr
\hline
          $B_n$, $n\geq 2$&$\om_1,\om_n$  \cr
\hline
$B_2$, $p\neq 2,3$&$\om_1,\om_2, \frac{p-3}{2}\om_1+\om_2,\frac{p-1}{2}\om_1$
\cr\hline
          $C_n$, $n>1$, $p=2$&$\om_1,\om_n$\cr  \hline
$C_n$, $n\geq 3$, $p\neq   2$&$\om_1,
\om_{n-1}+\frac{p-3}{2}\om_n,\frac{p-1}{2}\om_n$\cr
\hline $C_2$, $p\neq 2,3$&$\om_1,\om_2,
\om_1+\frac{p-3}{2}\om_2,\frac{p-1}{2}\om_2$\cr \hline
$C_3$, $p\neq   2$&$\om_3$\cr
\hline
$D_n$, $n>3$&$\om_1,\om_{n-1}, \om_n$\cr \hline

          $E_6$&$\om_1$, $\om_6$\cr
\hline
          $E_7$&$\om_7$  \cr
\hline
          $F_4 $, $p=3$&$\om_4$ \cr
\hline
          $ G_2$, $p\neq 3$&$\om_1$
             \\
\hline $ G_2$, $p= 3$&$\om_1,\om_2$
             \\
\hline   \end{tabular}
}

\end{center}
\end{table}


\end{document}